\documentclass[12pt]{amsart}
\usepackage{amssymb,amstext,amsfonts,amsthm}
\usepackage{color}
\usepackage[mathscr]{euscript}
\usepackage{color}
\makeatletter
\@addtoreset{equation}{section}

\makeatother
\catcode`\@=11
\newcommand\lsection{\@startsection {section}{1}{\z@}%
                                   {-3.5ex \@plus -1ex \@minus -.2ex}%
                                   {1.0ex \@plus.2ex}%
                                   {\normalfont\large\bfseries}}
\catcode`\@=12
\newtheorem{thm}{Theorem}[section]
\newtheorem{cor}[thm]{Corollary}
\newtheorem{lem}[thm]{Lemma}
\newtheorem{prop}[thm]{Proposition}

\theoremstyle{definition}
\newtheorem{ex}[thm]{Example}
\newtheorem{defn}[thm]{Definition}
\newtheorem{rem}[thm]{Remark}

\newtheorem{ques}[thm]{Question}
\newcommand{\pd}[2]{\frac{\partial#1}{\partial#2}}

\newcommand{\lct}{\mathop{\mathrm{lct}}\nolimits}
\newcommand{\ord}{\mathop{\mathrm{ord}}\nolimits}
\newcommand{\GL}{\mathop{\mathrm{GL}}\nolimits}
\newcommand{\C}{{\mathbb{C}}}
\newcommand{\R}{{\mathbb{R}}}
\newcommand{\Q}{{\mathbb{Q}}}
\newcommand{\Z}{{\mathbb{Z}}}
\newcommand{\V}{{\mathrm{V}}}
\newcommand{\supp}{{\mathrm{supp}}}
\newcommand{\m}{{\mathbf{m}}}
\newcommand{\LL}{{\mathcal{L}}}
\def\k{\mathit k}
\def\geq{\geqslant}
\def\leq{\leqslant}
\renewcommand{\O}{{\mathcal{O}}}
\newcommand{\comp}{\raisebox{0.1ex}{\hspace{0ex}
{$\scriptstyle{\circ}$}\hspace{0.5ex}}}         
\def\*{\color{green}\blacksquare\hspace{-1mm}\blacksquare\hspace{-1mm}\blacksquare}
\allowdisplaybreaks[4]
\selectfont

\begin{document}
\title[Mixed \L ojasiewicz exponents and log canonical thresholds]
{Mixed \L ojasiewicz exponents, log canonical \\ \vskip4pt thresholds
of ideals and bi-Lipschitz equivalence}

\author{Carles Bivi\`a-Ausina}
\address{
Institut Universitari de Matem\`atica Pura i Aplicada,
Universitat Polit\`ecnica de Val\`encia,
Cam\'i de Vera, s/n,
46022 Val\`encia, Spain}
\email{carbivia@mat.upv.es}

\author{Toshizumi Fukui}
\address{Department of Mathematics,
Saitama University,
255 Shimo-Okubo, Sakura-ku
Saitama 338-8570, Japan}
\email{tfukui@rimath.saitama-u.ac.jp}


\thanks{The first author was partially supported by DGICYT
Grant MTM2009--08933.}


\begin{abstract}
We study the \L ojasiewicz exponent and the log canonical threshold of ideals of $\O_n$
when restricted to generic subspaces of $\C^n$ of different dimensions. We obtain effective formulas
of the resulting numbers for ideals with monomial integral closure. An inequality relating these
numbers is also proven. We also introduce the notion of bi-Lipschitz equivalence of ideals and we
prove the bi-Lipschitz invariance of \L ojasiewicz exponents and log canonical thresholds of ideals.
\end{abstract}

\maketitle


\lsection{Introduction}


In 1970, O.\,Zariski posed in \cite[p.\,483]{Zariski} the following celebrated question:
\begin{center}
\begin{minipage}{140mm}
Let $f$ and $g$ be two holomorphic function germs $(\C^n,0)\to(\C,0)$.
If there is a homeomorhism $\varphi:(\C^n,0)\to(\C^n,0)$
so that $\varphi(f^{-1}(0))=g^{-1}(0)$, then
do the germs $f$ and $g$ have the same multiplicity?
\end{minipage}
\end{center}
This question is still unsolved except for the case $n=2$
and is known as the {\it Zariski's multiplicity conjecture}
(see the survey \cite{Eyral}).
One of the main difficulties to attack this question comes from the fact
that the image of a line by a homeomorphism $\varphi:(\C^n,0)\to (\C^n,0)$ may not carry
any algebraic (or analytic) structure.

Let $\O_n$ denote the ring of complex analytic function germs $(\C^n,0)\to \C$ and let $\m_n$ denote
the maximal ideal of $\O_n$. We recall that the {\it multiplicity}
or {\it order} of $f$ is defined as the maximum of those $r\in\Z_{\geq 1}$ such that $f\in \m_n^r$.


Let $f\in\O_n$ such that $f$ has an isolated singularity at the origin. In his famous book \cite{Milnor},
J.\,Milnor showed several topological interpretations of the number
$$
\mu(f)=\dim_{\C}\O_n/\langle\textstyle{
\pd{f}{x_1},\dots,\pd{f}{x_n}}\rangle,
$$
which is usually known as the Milnor number of $f$.
Zariski's multiplicity conjecture and Milnor's book have been
some of the most important motivations of many researchers
to explore the relations between invariants of different nature
(topological, analytic or algebraic) of a given singular function germ $f\in\O_n$, or more generally,
of complete intersection singularities.

B.\,Teissier introduced in \cite[p.\,300]{Cargese} the sequence of Milnor numbers
$$
\mu^*(f)=\left(\mu^{(n)}(f),\mu^{(n-1)}(f),\dots,\mu^{(1)}(f)\right)
$$
where $\mu^{(i)}(f)$ denotes the Milnor number of the restriction of $f$ to a generic linear $i$-dimensional
subspace of $\C^n$, for $i\in\{0,1,\dots, n\}$. In particular $\mu^{(1)}(f)=\ord(f)-1$ and $\mu^{(n)}(f)=\mu(f)$.
By the results of Teissier \cite[p.\,334]{Cargese} and Brian\c con-Speder \cite[p.\,159]{BS2} we know that,
if $f_t:(\C^n,0)\to (\C,0)$ denotes an analytic family of function germs such that
$f_t$ have simultaneously isolated singularities at $0$, then the constancy of $\mu^*(f_t)$ is equivalent to the
Whitney equisingularity of the deformation $f_t$.
In \cite[1.7]{Teissier77}, Teissier
also obtained a relation between the set of polar multiplicities of a given function germ $f\in\O_n$ with the
\L ojasiewicz exponent $\LL_0(\nabla f)$. The number $\LL_0(\nabla f)$ is defined as the infimum of those
$\alpha\in\R_{\geq 0}$ for which there exists a positive constant $C>0$ and an open neighbourhood $U$ of $0\in\C^n$
such that
$$
\|x\|^\alpha\leq C\|\nabla f(x)\|
$$
for all $x\in U$, where $\nabla f$ denotes the gradient map $(\pd{f}{x_1},\dots,\pd{f}{x_n})$ of $f$.
Teissier also asked in \cite[p.\,287]{Teissier77} whether $\LL_0(\nabla f_t)$ remains constant
in $\mu$-constant analytic deformations $f_t:(\C^n,0)\to (\C,0)$. There is still no general answer to this question.
However as a consequence of \cite[1.7]{Teissier77} and \cite[Th\'eor\`eme\,6]{Teissier77} it follows that, if $f_t:(\C^n,0)\to (\C,0)$
denotes a $\mu^*$-constant analytic deformation, then $\LL_0(\nabla f_t)$ is also constant.


The research of such invariants is motivated not only to understand the topology of hypersurfaces and singular varieties
in general but also to understand the behaviour of functions and maps.
In \cite{Mather}, J.\,Mather introduced the language to investigate singularities of maps and functions.
This language has been widely-accepted and studied (see for instance the survey of C.\,T.\,C.\,Wall \cite{Wall}).
J.\,Mather defined the notions of right equivalence, right-left equivalence and
contact equivalence for map germs. The corresponding equivalence classes are the orbits of the action of the groups
$\mathcal R$, $\mathcal A$ and $\mathcal K$ respectively, where
\begin{itemize}
\item $\mathcal R$ is the group of diffeomorphism germs of the source,
\item $\mathcal A$ is the direct product of the group of diffeomorphism germs of the source and the target,
\item $\mathcal K$ is the group that is formed by the elements
$(\varphi(x),\phi_x(y))$ so that
\begin{itemize}
\item[{\scriptsize $\bullet$}]
$x\mapsto\varphi(x)$ is a diffeomorphism germ of the source, and
\item[{\scriptsize $\bullet$}]
$y\mapsto\phi_x(y)$ are diffemorphism germs of the target for any
$x$.
\end{itemize}
\end{itemize}
In \cite[(2.3)]{Mather}, J.\,Mather also showed that two map germs $f$ and $g$ are
contact equivalent if and only if
the ideals generated by the component functions of $f$ and that of $g\comp\varphi$, respectively,
are the same for some coordinate change $\varphi$ of the source.
These notions have clearly a holomorphic analogue.
For shortness, we often call right equivalence, right-left equivalence and contact equivalence
by $\mathcal R$-equivalence, $\mathcal A$-equivalence, and
$\mathcal K$-equivalence, respectively.

It is natural to consider the bi-Lipschitz analogue of these notions.
This direction seems to be first considered in \cite{RT} by J.-J.\,Risler
and D.\,Trotman in the context of singularity theory
after the establishment of the theory of Lipschitz stratifications \cite{Mostowski}
(see also \cite{Parusinski}).
They showed that if two holomorphic function germs are right-left
equivalent in the bi-Lipschitz sense, then they have the same multiplicity.
This fact was a bit surprising, since there is a bi-Lipschitz homeomorphism
which sends a semi-line to the log spiral:
$$
(\R^2,0)\to(\R^2,0),\qquad
(r,\theta)\to(r,\theta-\log r),\quad\textrm{ in terms of polar coordinates $(r,\theta)$.}
$$
The images of lines by bi-Lipschitz homeomorphisms
may not be analytic spaces, but the concept of bi-Lipschitz homeomorphism is substantially more fruitful than just talking
about homeomorphisms. After \cite{RT}, researchers in singularity theory started
to investigate singularities from the viewpoint of bi-Lipschitz equivalence
in several contexts. According to this, we list (non-exhaustively) the following topics of study and some references:
\begin{itemize}
\item
bi-Lipschitz $\mathcal R$-classification of functions
(\cite{FR,HP1,HP2,RV})
\item
properties on bi-Lipschitz $\mathcal K$-equivalence (\cite{BCFR,RV})
\item
classification of complex surfaces singularities in the bi-Lipschitz context (\cite{BFN})
\item
directional properties of subanalytic sets via bi-Lipschitz homeomorphisms (\cite{KP})
\item
bi-Lipschitz stratifications (\cite{JV,V})
\item
the notion of integral closure technique in the bi-Lipschitz context (\cite{Gaffney}).
\end{itemize}


One of the motivations of this paper is the study of the invariance of $\LL_0(\nabla f)$ under bi-Lipschitz equivalences
(see Subsection \ref{BLeq} and Theorem \ref{BLInvFunc}) and related outcomes of the discussion based on the estimation
of \L ojasiewicz exponents. Moreover, we explore in \S\ref{MLoj}, \S\ref{genHickel} and \S\ref{MLojmonomial} the notion of
\L ojasiewicz exponent $\LL_0(I_1,\dots, I_n)$ of $n$ ideals in a Noetherian local ring of dimension $n$.
This concept was introduced in \cite{BiviaMZ2} using the notion of mixed multiplicities of ideals.
If $I$ denotes an ideal of finite colength of $\O_n$, then
we are particularly interested in the \L ojasiewicz exponent that arises when restricting $I$ to generic
linear subspaces of $\C^n$ of different dimensions, thus leading to the sequence of relative \L ojasiewicz exponents (see Definition
\ref{relativeLE}).

The notion of mixed multiplicities of ideals was originated by the results of
Risler and Teissier in \cite{Cargese} about the study of the $\mu^*$-sequence of
function germs with an isolated singularity at the origin.
Subsequently there is a well-developed theory of the notion of
mixed multiplicities of ideals which can be
found in \cite{HS} (see also the invaluable paper of D.\,Rees \cite{Rees2}).


In \S\ref{genHickel}, we discuss a generalization of an inequality proven by Hickel \cite{Hickel}.
In \S\ref{MLojmonomial} we obtain an expression of the sequence of relative \L ojasiewicz exponents of a monomial ideal $I$ of
$\O_n$ in terms of the Newton polyhedron of $I$.
In \S\ref{biLipLoj}, we show the bi-Lipschitz $\mathcal A$-invariance of $\LL_0(\nabla f)$ and several outcomes of the proof.
We also show a result about the constancy of \L ojasiewicz exponents in $\mu$-constant deformations of weighted homogeneous functions $(\C^n,0)\to (\C,0)$.
In \S\ref{LogCanTh}, we discuss the notion of log canonical threshold $\lct(I)$ of an ideal $I$ of $\O_n$.
We show that this number is bi-Lipschitz invariant and show a relation between $\lct(I)$ and \L ojasiewicz exponents that
enables us to express $\lct(I)$ in terms of
\L ojasiewicz exponents when the integral closure $\overline I$ of $I$ is a monomial ideal.
In \S\ref{LogCanTh2} we discuss the behaviour of $\lct(I)$
when restricting $I$ to generic
$i$-dimensional linear subspaces of $\C^n$, for $i=1,\dots, n$. Then there arises the sequence
$\lct^*(I)=(\lct^{(n)}(I),\dots,\lct^{(1)}(I))$ for which we show a closed formula when
$\overline I$ is a monomial ideal.

The authors would like to thank S.\,Ishii, M.\,A.\,S.\,Ruas and D.\,Trotman for helpful conversations.


\lsection{Preliminaries}

We start by recalling notational conventions.
Let $a(x)$ and $b(x)$ be two function germs $(\C^n,x_0)\to \R$, where $x_0\in\C^n$. Then
\begin{itemize}
\item {\it $a(x)\lesssim b(x)$ near $x_0$} means that there exists a positive constant $C>0$ and an open
neighbourhood $U$ of $x_0$ in $\C^n$ such that $a(x)\leq C\,b(x)$, for all $x\in U$.
\item {\it $a(x)\sim b(x)$ near $x_0$} means that $a(x)\lesssim b(x)$ near $x_0$ and
$b(x)\lesssim a(x)$ near $x_0$.
\end{itemize}
For an $n$-tuple $x=(x_1,\dots,x_n)\in\C^n$, we write
$\|x\|=\sqrt{|x_1|^2+\cdots+|x_n|^2}$.

\subsection{Bi-Lipschitz equivalences}\label{BLeq}
We start with recalling the definition of bi-Lipschitz map.
A map germ $f:(\C^n,0)\to(\C^p,0)$ is said to be {\it Lipschitz} if
$$
\|f(x)-f(x')\|\lesssim\|x-x'\| \textrm{ near 0}.
$$
We say that a homeomorphism $h:(\C^n,0)\to(\C^n,0)$ is {\it bi-Lipschitz} if
$h$ and $h^{-1}$ are Lipschitz.
Now we can state obvious bi-Lipschitz analogues for several equivalence relations:
\begin{itemize}
\item Two map germs $f,g:(\C^n,0)\to(\C^p,0)$ are said to be
{\it bi-Lipschitz $\mathcal R$-equivalent}
if there is a bi-Lipschitz homeomorphism
$\varphi:(\C^n,0)\to(\C^n,0)$ so that $f=g\comp\varphi$.
\item
Two map germs $f,g:(\C^n,0)\to(\C^p,0)$ are said to be
{\it bi-Lipschitz $\mathcal A$-equivalent}
if there are a bi-Lipschitz homeomorphism
$\varphi:(\C^n,0)\to(\C^n,0)$ and a bi-Lipschitz homeomorphism $\phi:(\C^p,0)\to(\C^p,0)$
so that $\phi(f(x))=g(\varphi(x))$, for all $x$
belonging to some open neighbourhood of $0\in\C^n$.
\item
Two map germs $f,g:(\C^n,0)\to(\C^p,0)$ are said to be
{\it bi-Lipschitz $\mathcal K$-equivalent}
if there are a bi-Lipschitz homeomorphism
$\varphi:(\C^n,0)\to(\C^n,0)$ and a
bi-Lipschitz homeomorphism $\Phi:(\C^n\times\C^p,0)\to(\C^n\times\C^p,0)$, written as
$(x,y)\mapsto(\varphi(x),\phi_x(y))$,
so that $\Phi(\C^n\times\{0\})=\C^n\times\{0\}$ and $\phi_x(f(x))=g(\varphi(x))$, for all $x$ belonging to
some open neighbourhood of $0\in\C^n$.
\item
Two map germs $f,g:(\C^n,0)\to(\C^p,0)$ are said to be
{\it bi-Lipschitz $\mathcal K^*$-equivalent}
if there are a bi-Lipschitz homeomorphism
$\varphi:(\C^n,0)\to(\C^n,0)$ and a map $A:(\C^n,0)\to\GL(\C^p)$
so that $A(x)$ and $A(x)^{-1}$ are Lipschitz and that
$A(x)f(x)=g(\varphi(x))$,
for all $x$ belonging to some open neighbourhood of $0\in\C^n$.
\item
Two subsets $X_1$ and $X_2$ of $(\C^n,0)$ are
{\it bi-Lipschitz equivalent} if
there is a bi-Lipschitz homeomorphism
$\varphi:(\C^n,0)\to(\C^n,0)$ so that $\varphi(X_1)=X_2$.
\end{itemize}
The definition of bi-Lipschitz $\mathcal K$-equivalence
is used in \cite{BCFR}.
It is possible to consider a weaker version of the
definition of $\mathcal K$-equivalence by
replacing the condition that $\Phi$ is bi-Lipschitz
by the condition that $\phi_x$ is bi-Lipschitz, for all $x$ belonging to some open neighbourhood of $0\in\C^n$.
We only need this condition in the proof of Theorem \ref{lctBL}.

The definition of $\mathcal K^*$-equivalence is inspired by the
condition (iii) of the first proposition in paragraph (2.3) in \cite{Mather}.

For a bi-Lipschitz homeomorphism $\varphi:(\C^n,0)\to(\C^n,0)$,
we do not have the induced map $\varphi^*:\O_n\to\O_n$,
since $f\circ\varphi$ may not be holomorphic for $f\in\O_n$.
So we introduce the following definition.

\begin{defn}\label{BLideals}
Let $I$ and $J$ be ideals of $\O_n$.
We say that $I$ and $J$ are {\it bi-Lipschitz equivalent} if
there exist two families $f_1,\dots,f_p$ and $g_1,\dots,g_q$ of functions of $\O_n$ such that
\begin{enumerate}
\item[(a)] $\langle f_1,\dots,f_p\rangle\subseteq I$ and $\overline{\langle f_1,\dots,f_p\rangle}=\overline I$,
\item[(b)] $\langle g_1,\dots,g_q\rangle\subseteq J$ and $\overline{\langle g_1,\dots,g_q\rangle}=\overline J$,
\item[(c)]
there is a bi-Lipschitz homeomorphism $\varphi:(\C^n,0)\to(\C^n,0)$ such
that
$$
\|(f_1(x),\dots,f_p(x))\|\sim\|(g_1(\varphi(x)),\dots,g_q(\varphi(x)))\|
\qquad\textrm{ near } 0.
$$
\end{enumerate}
\end{defn}

We remark that, under the conditions of item (a), the ideal $\langle f_1,\dots,f_p\rangle$ is
usually called a {\it reduction} of $I$ (see \cite[p.\,6]{HS}).



Here there are some obvious consequences:
\begin{itemize}
\item
If two map germs $f,g:(\C^n,0)\to (\C^p,0)$ are bi-Lipschitz $\mathcal R$-equivalent,
then they are bi-Lipschitz $\mathcal A$ (and $\mathcal K^*$)-equivalent.
\item
If two map germs $f,g:(\C^n,0)\to (\C^p,0)$ are bi-Lipschitz
$\mathcal A$-equivalent or $\mathcal K^*$-equivalent,
then they are bi-Lipschitz $\mathcal K$-equivalent.
\item
If two map germs $f$ and $g$ are bi-Lipschitz $\mathcal K$-equivalent,
then the ideals generated by their components are bi-Lipschitz equivalent.
\item
If two ideals are bi-Lipschitz equivalent,
then their zero loci are bi-Lipschitz equivalent.
\end{itemize}
The following questions seem to be open.
\begin{ques}
\begin{itemize}
\item
If $f$ and $g$ are bi-Lipschitz $\mathcal K$-equivalent,
are $f$ and $g$ bi-Lipschitz $\mathcal K^*$-equivalent?
\item
If $f$ and $g$ are bi-Lipschitz $\mathcal A$-equivalent,
are $f$ and $g$ bi-Lipschitz $\mathcal K^*$-equivalent?
\end{itemize}
\end{ques}
\begin{ques}
Let $X$ and $Y$ be germs of complex analytic subvarieties at $0$ in
$\C^n$.
If there exist a bi-Lipschitz homeomorphism
$h:(\C^n,0)\to(\C^n,0)$ so that $h(X)=Y$, are the respective defining ideals
of $X$ and $Y$ bi-Lipschitz equivalent?
\end{ques}
Let $f,g:(\C^n,0)\to(\C,0)$ be two holomorphic functions.
Assume that there is a bi-Lipschitz homeomorphism
$\varphi:(\C^n,0)\to(\C^n,0)$ so that $f^{-1}(0)=\varphi(g^{-1}(0))$.
The authors do not know whether
${g(\varphi(x))}/{f(x)}$ is bounded away from $0$ and
infinity, or not.

\subsection{\L ojasiewicz exponent of ideals}
Let $I$ and $J$ be ideals of $\O_n$. Let
$\{f_1,\dots,f_p\}$ be a generating system of $I$ and
let $\{g_1,\dots,g_q\}$ be a generating system of $J$.
Let us consider the maps
$f=(f_1,\dots,f_p):(\C^n,0)\to (\C^p,0)$ and $g=(g_1,\dots,g_q):(\C^n,0)\to (\C^q,0)$.
We define the {\it \L ojasiewicz exponent of $I$ with respect to $J$}, denoted by
$\mathcal L_J(I)$, as the infinimum of the set
\begin{equation}\label{defLJI}
\big\{\alpha\in\R_{\geq 0}: \|g(x)\|^\alpha\lesssim\|f(x)\| \textrm{ near $0$}\big\}.
\end{equation}
By convention, we set $\inf \emptyset=\infty$. So if the previous set is empty, then we
set $\LL_J(I)=\infty$.

We thus have that $\LL_J(I)$ is finite if and only if $V(I)\subseteq V(J)$ (see \cite{LT}).

Let us suppose that the ideal $I$ has finite colength.
When $J=\m_n$, then we denote the number $\LL_J(I)$ by $\LL_0(I)$. That is
$$
\mathcal L_0(I)=
\inf\big\{\alpha\in\R_{\geq 0}:\|x\|^\alpha\lesssim\|f(x)\|\textrm{ near $0$}\big\}.
$$
We refer to $\LL_0(I)$ as the {\it \L ojasiewicz exponent of $I$}.

\lsection{The sequence of mixed \L ojasiewicz exponents}\label{MLoj}

If $I$ denotes an ideal of a ring $R$, then we denote by $\overline I$ the integral closure of $I$.
Let us suppose that $I$ is an ideal of finite colength of $\O_n$ and let $J$ be a proper ideal of $\O_n$.
Then, by virtue of the results of Lejeune and Teissier in \cite[Th\'eor\`eme 7.2]{LT},
the \L ojasiewicz exponent $\LL_J(I)$ can be expressed algebraically as
$$
\LL_J(I)=\inf\left\{\frac rs: r,s\in\Z_{\geq 1},\, J^r\subseteq\overline{I^s}\right\}.
$$

This fact is one of the motivations of the definition in \cite{BiviaMZ2} of the notion of \L ojasiewicz
exponent of a set of ideals. The main tool used for this definition is the mixed multiplicity of $n$ ideals in
a local ring of dimension $n$.

Let $(R,\m)$ denote a Noetherian local ring of dimension $n$. If $I_1,\dots, I_n$ are ideals of $R$ of finite colength,
then we denote by $e(I_1,\dots, I_n)$ the mixed multiplicity of $I_1,\dots, I_n$ defined by
Teissier and Risler in \cite[\S2]{Cargese}. We also refer to \cite[\S17.4]{HS} or \cite{Swanson} for the definitions and fundamental
results concerning mixed multiplicities of ideals. Here we recall
briefly the definition of $e(I_1,\dots, I_n)$. Under the conditions
exposed above, let us consider the function $H:\Z^n_{\geq 0}\to\Z_{\geq 0}$
given by
\begin{equation}\label{Hilbertfunct}
H(r_1,\dots, r_n)=\ell\left(\frac{R}{I_1^{r_1}\cdots
I_n^{r_n}}\right),
\end{equation}
for all $(r_1,\dots, r_n)\in\Z^n_{\geq 0}$, where $\ell(M)$ denotes the length of a given $R$-module $M$. Then, it is proven in
\cite{Cargese} that there exists a polynomial $P(x_1,\dots, x_n)\in \Q[x_1,\dots, x_n]$ of
degree $n$ such that
$$
H(r_1,\dots, r_n)=P(r_1,\dots, r_n),
$$
for all sufficiently large $r_1,\dots, r_n\in\Z_{\geq 0}$. Moreover, the
coefficient of the monomial $x_1\cdots x_n$ in $P(x_1,\dots, x_n)$ is an integer.
This integer is called the {\it mixed multiplicity} of $I_1,\dots,
I_n$ and is denoted by $e(I_1,\dots, I_n)$.

We remark that if $I_1,\dots, I_n$ are all equal to a given ideal
$I$ of finite colength of $R$, then $e(I_1,\dots, I_n)=e(I)$, where $e(I)$ denotes the
Samuel multiplicity of $I$. If $i\in\{0,1,\dots, n\}$, then we denote by $e_i(I)$ the mixed
multiplicity $e(I,\dots, I, \m,\dots, \m)$, where $I$ is repeated $i$ times and the maximal
ideal $\m$ is repeated $n-i$ times. In particular $e_n(I)=e(I)$ and $e_0(I)=e(\m)$.

If $f\in\O_n$ is an analytic function germ with an isolated singularity at the origin and
$J(f)$ denotes the Jacobian ideal of $f$, then we denote by $\mu^{(i)}(f)$
the Milnor number of the restriction of $f$ to a generic linear subspace of dimension $i$ passing through the origin
in $\C^n$, for $i=0,1,\dots, n$. Teissier showed in \cite{Cargese} that $\mu^{(i)}(f)=e_i(J(f))$, for all $i=0,1,\dots, n$.
The {\it $\mu^*$-sequence of $f$} is defined as $\mu^*(f)=(\mu^{(n)}(f),\dots, \mu^{(1)}(f))$.

If $g_1,\dots, g_r\in R$ and they generate an ideal $J$ of $R$ of finite colength then we denote
the multiplicity $e(J)$ also by $e(g_1,\dots, g_r)$. We will need the
following known result (see for instance \cite[p.\ 345]{HS}).

\begin{lem}\label{fitemlesmix}
Let $(R,\m)$ be a Noetherian local ring of dimension $n\geqslant 1$. Let $I_1,\dots,
I_n$ be ideals of $R$ of finite colength. Let $g_1,\dots, g_n$ be
elements of $R$ such that $g_i\in I_i$, for all $i=1,\dots, n$,
and the ideal $\langle g_1,\dots, g_n\rangle$ has also finite
colength. Then
$$
e(g_1,\dots, g_n)\geqslant e(I_1,\dots, I_n).
$$
\end{lem}



\begin{defn}\label{lasigma} Let $(R,\m)$ be a Noetherian local ring of dimension $n$. Let $I_1,\dots,
I_n$ be ideals of $R$. Then we define
\begin{equation}\label{sigma}
\sigma(I_1,\dots, I_n)=\max_{r\in\Z_{\geq 1}}\,e(I_1+\m^r,\dots,
I_n+\m^r).
\end{equation}
\end{defn}

The set of integers $\{e(I_1+\m^r,\dots, I_n+\m^r): r\in\Z_{\geq 0}\}$ is
not bounded in general. Thus $\sigma(I_1,\dots,I_n)$ is not
always finite. The
finiteness of $\sigma(I_1,\dots, I_n)$ is characterized in
Proposition \ref{sigmaexists}. We remark that if $I_i$ has finite
colength, for all $i=1,\dots, n$, then $\sigma(I_1,\dots, I_n)$
equals the usual notion of mixed multiplicity $e(I_1,\dots, I_n)$.

Let us suppose that the residue field $k=R/\m$ is infinite. Let $I_1,\dots, I_n$ be ideals of $R$.
We say that a given property is satisfied for a {\it sufficiently general element of $I_1\oplus\cdots\oplus I_n$}, when, after identifying $(I_1/\m I_1)\oplus \cdots \oplus (I_n/\m I_n)$
with $k^s$, for some $s\geq 1$, there exist a Zariski open subset $U\subseteq k^s$ such that the said
property holds for all elements of $U$.

\begin{prop}[{\cite[p.\,393]{BiviaMRL}}]\label{sigmaexists}
Let $I_1,\dots, I_n$ be ideals of a Noetherian
local ring $(R,\m)$ such that the residue field $\k=R/\m$ is
infinite. Then $\sigma(I_1,\dots, I_n)<\infty$ if and only if
there exist elements $g_i\in I_i$, for $i=1,\dots, n$, such that
$\langle g_1,\dots, g_n\rangle$ has finite colength. In this case,
we have that $\sigma(I_1,\dots, I_n)=e(g_1,\dots, g_n)$ for a
sufficiently general element $(g_1,\dots, g_n)\in I_1\oplus
\cdots \oplus I_n$.
\end{prop}

Proposition \ref{sigmaexists} shows that, if $\sigma(I_1,\dots,
I_n)<\infty$, then $\sigma(I_1,\dots, I_n)$ is equal to the mixed
multiplicity of $I_1,\dots,I_n$ defined by Rees in \cite[p.\,181]{Reesllibre} (see also \cite{RS})
via the notion of general extension of a local ring. Therefore, we
will refer to $\sigma(I_1,\dots, I_n)$ as the {\it Rees' mixed
multiplicity} of $I_1,\dots, I_n$.

\begin{lem}[{\cite[p. 392]{BiviaMZ2}}]\label{reverseincl}
Let $(R,\m)$ be a Noetherian local ring of dimension $n\geq 1$. Let
$J_1,\dots, J_n$ be ideals of $R$ such that
$\sigma(J_1,\dots, J_n)<\infty$. Let $I_1,\dots, I_n$ be ideals of
$R$ for which $J_i\subseteq I_i$, for all $i=1,\dots, n$. Then
$\sigma(I_1,\dots, I_n)<\infty$ and
$$
\sigma(J_1,\dots, J_n)\geqslant \sigma(I_1,\dots, I_n).
$$
\end{lem}

Under the conditions of Definition \ref{lasigma}, let us denote
by $J$ a proper ideal of $R$. From Lemma \ref{reverseincl} we obtain easily that
$$
\sigma(I_1,\dots, I_n)=\max_{r\in\Z_{\geq 0}}\,\sigma(I_1+J^r,\dots, I_n+J^r).
$$
Let us suppose that $\sigma(I_1,\dots, I_n)<\infty$. Hence, we define
\begin{equation}\label{laerresubJ}
r_J(I_1,\dots, I_n)=\min\big\{r\in\Z_{\geq 0}: \sigma(I_1,\dots,
I_n)=\sigma(I_1+J^r,\dots, I_n+J^r)\big\}.
\end{equation}

If $I$ is an ideal of finite colength of $R$
then we denote $r_J(I,\dots, I)$ by $r_J(I)$.
We remark that if $R$ is quasi-unmixed, then, by the Rees'
multiplicity theorem (see for instance \cite[p.\,222]{HS}) we have
$$
r_J(I)=\min\big\{r\in\Z_{\geq 0}: J^r\subseteq \overline{I}\big\}.
$$
We will denote the integer $r_\m(I)$ by $r_0(I)$.

\begin{defn}[\cite{BE2012}]\label{DefLJI}
Let $(R,\m)$ be a Noetherian local ring of dimension $n$. Let $I_1,\dots,
I_n$ be ideals of $R$ such that $\sigma(I_1,\dots, I_n)<\infty$. Let
$J$ be a proper ideal of $R$.
We define the {\it {\L}ojasiewicz exponent of $I_1,\dots, I_n$ with
respect to $J$}, denoted by $\LL_J(I_1,\dots, I_n)$, as
\begin{equation}\label{LJI}
\LL_J(I_1,\dots, I_n)=\inf_{s\geq 1}\frac{r_J(I_1^s,\dots, I_n^s)}{s}.
\end{equation}
\end{defn}

In accordance with mixed multiplicities of ideals, we also refer to the number $\LL_J(I_1,\dots, I_n)$
as the {\it mixed \L ojasiewicz exponent of $I_1,\dots, I_n$ with respect to $J$}; when $J=\m$ we denote this
number by $\LL_0(I_1,\dots, I_n)$.

\begin{rem}\label{LJIhabitual}
Let us observe that, under the conditions of Definition \ref{DefLJI}, if $I$ is an ideal of finite colength of $R$
such that $I_1=\cdots=I_n=I$, then the right hand side of (\ref{LJI}) can be rewritten as
\begin{equation}\label{sesimplifica}
\inf\left\{\frac rs: r,s\in\Z_{\geq 1},\, e(I^s)=e(I^s+J^r)\right\}.
\end{equation}
If we assume that $R$ is quasi-unmixed and $r,s\in\Z_{\geq 1}$, then the condition $e(I^s)=e(I^s+J^r)$
is equivalent to saying that $J^r\subseteq \overline{I^s}$, by the Rees' multiplicity theorem. Therefore (\ref{sesimplifica}) is expressed as
$$
\inf\left\{\frac rs: r,s\in\Z_{\geq 1},\, J^r\subseteq \overline{I^s}\right\},
$$
which coincides with the usual notion of \L ojasiewicz exponent $\LL_J(I)$ of $I$ with respect to $J$ (see \cite[Th\'eor\`eme 7.2]{LT}).
\end{rem}

As a particular case of the previous definition we introduce the following concept.

\begin{defn}\label{relativeLE} Let $(R,\m)$ be a Noetherian local ring of dimension $n$. Let $I$ be an
ideal of $R$ of finite colength. If $i\in\{1,\dots, n\}$, then we define
the {\it $i$-th relative \L ojasiewicz exponent of $I$} as
\begin{equation}\label{LiI}
\LL_0^{(i)}(I)=\LL_0(\underbrace{I,\dots, I}_{\textnormal{$i$ times}},\underbrace{\m,\dots, \m}_{\textnormal{$n-i$ times}}).
\end{equation}

We define the {\it $\LL^*_0$-vector}, or {\it $\LL^*_0$-sequence}, of $I$ as
$$
\LL^*_0(I)=\big(\LL_0^{(n)}(I),\dots, \LL_0^{(1)}(I)\big).
$$
If $J$ denotes a proper ideal of $R$, then we define {\it the $i$-th relative \L ojasiewicz exponent of $I$ with respect to $J$},
denoted by $\LL_J^{(i)}(I)$, by replacing $\m$ by $J$ in (\ref{LiI}). The {\it $\LL^*_J$-sequence} of $I$ is defined analogously.
\end{defn}

\begin{defn}
Let $(X,0)\subseteq (\C^n,0)$ be the germ at $0$ of a complex analytic variety $X$.
Let $h_1,\dots, h_m\in\O_n$ such that $(X,0)=V(h_1,\dots, h_m)$. Let $h$ denote the
map $(h_1,\dots, h_m):(\C^n,0)\to (\C^m,0)$.
Let $I$ be an ideal of $\O_n$ such that $V(I)\cap X=\{0\}$.
Then we define the {\it \L ojasiewicz exponent of $I$ relative to $(X,0)$} as the infimum
of those $\alpha>0$ such that there exists a constant $C>0$ and an open neighbourhood $U$ of
$0\in\C^n$ such that $\Vert x\Vert^\alpha\leq C\Vert h(x)\Vert$,
for all $x\in U\cap X$.
\end{defn}

By the results of Lejeune-Teissier \cite{LT} we have that if $J$
is the ideal of $\O_n$ generated by $h_1,\dots, h_m$, then $\LL_{(X,0)}(I)=\LL_J(I)$.

We will study the number $\LL_{(X,0)}(I)$ specially when $(X,0)$ is a linear subspace of $\C^n$.

\begin{thm}
Let $\pi:M\to\C^n$ be a proper modification so that $\pi^*(\m I)_0$
is formed by normal crossing divisors whose support has the
irreducible decomposition $\cup_iD_i$.
If
$$
(\pi^*\m)_0=\sum_is_iD_i,\quad (\pi^*I)_0=\sum_im_iD_i,
\quad s_i, m_i\in\mathbb{Z},
$$
then we have
\begin{equation}\label{L_X}
\LL_{(X,0)}(I)=\max\left\{\frac{m_i}{s_i}:D_i\cap
 X'\ne\emptyset\right\}
\end{equation}
where $X'$ denotes the strict transform of $X$ by $\pi$
(see \cite{BE} for details).
\end{thm}

\lsection{Inequalities relating \L ojasiewicz exponents and mixed multiplicities}\label{genHickel}

This section is motivated by the results of Hickel in \cite{Hickel}. In this section we
show some results showing how \L ojasiewicz exponents are related with quotients of
mixed multiplicities; the main result in this direction is Theorem \ref{gralHickel}.

\begin{prop}\label{meollo}
Let $(R,\m)$ be a quasi-unmixed Noetherian local ring of dimension $n$.
Let $I_1,\dots, I_n, J$ be ideals of $R$ such that
$\sigma(I_1,\dots, I_n)<\infty$, $\sigma(I_1,\dots, I_{n-1}, J)<\infty$ and $I_n$ has finite colength.
Then
$$
\frac{\sigma(I_1,\dots, I_n)}{\sigma(I_1,\dots, I_{n-1}, J)}\leq \LL_J(I_n).
$$
\end{prop}

\begin{proof}
Let $r,s\in\Z_{\geq 1}$. Let us suppose that $J^r\subseteq\overline{I_n^s}$.
Then we obtain
\begin{align}
r\cdot\sigma(I_1,\dots, I_{n-1}, J)&=\sigma(I_1,\dots, I_{n-1}, J^r)\label{lowbound1}\\
&\geq \sigma(I_1,\dots, I_{n-1}, I_n^s)=s\cdot\sigma(I_1,\dots, I_{n-1}, I_n).\label{lowbound2}
\end{align}
We refer to \cite[Lemma 2.6]{BiviaMZ2} for equality (\ref{lowbound1}) and to Lemma \ref{fitemlesmix} for the
inequality in (\ref{lowbound2}).
In particular
$$
\frac rs\geq \frac{\sigma(I_1,\dots, I_{n-1}, I_n)}{\sigma(I_1,\dots, I_{n-1}, J)}.
$$
By \cite[Th\'eor\`eme 7.2]{LT} we have $\LL_J(I_n)=\inf\{\frac rs: r,s\in\Z_{\geq 1}, J^r\subseteq\overline{I_n^s}\}$ (see Remark \ref{LJIhabitual}).
Then the result follows.
\end{proof}

\begin{cor}
Let $(R,\m)$ be a quasi-unmixed Noetherian local ring of dimension $n$. Let $I$ be an ideal of finite colength of
$R$. Then
\begin{equation}\label{comparing}
\frac{e(I)}{e_{n-1}(I)}\leq \LL_0(I).
\end{equation}
and equality holds if and only if
$$
e_{n-1}(I)^ne(I)=e(I^{e_{n-1}(I)} + \m^{e(I)}).
$$
\end{cor}

\begin{proof}
Inequality (\ref{comparing}) follows from applying Proposition \ref{meollo} to the case $I_1=\cdots=I_n=I$ and $J=\m$.

By the definition of $\LL_0(I)$ we observe that equality holds in (\ref{comparing}) if and only if
$\m^{e(I)}\subseteq \overline{I^{e_{n-1}(I)}}$.
This inclusion is equivalent to saying that
$e(I^{e_{n-1}(I)})=e(I^{e_{n-1}(I)}+\m^{e(I)})$, by the Rees' multiplicity theorem.
\end{proof}

\begin{rem}
Let $w=(w_1,\dots, w_n)\in\Z_{\geq 1}^n$ and let $d\in\Z_{\geq1}$. Let us denote $\min_iw_i$ by $w_0$.
Let $f\in\O_n$ denote a semi-weighted homogeneous function germ of degree $d$ with respect to $w$.
It is known that $\LL_0(\nabla f)\leq \frac{d-w_0}{w_0}$ (see for instance \cite[Corollary 4.7]{BE2012}).
Hence it is interesting to determine when $\LL_0(\nabla f)$ attains the maximum possible value $\frac{d-w_0}{w_0}$
(see \cite{BE2012,KOP}).

By (\ref{comparing}) we obtain
\begin{equation}\label{interessant}
\frac{\mu(f)}{\mu^{(n-1)}(f)}\leq \LL_0(\nabla f).
\end{equation}
Therefore, if $\frac{\mu(f)}{\mu^{(n-1)}(f)}=\frac{d-w_0}{w_0}$ then we have the equality $\LL_0(\nabla f)=\frac{d-w_0}{w_0}$.

Let $f_t:(\C^3,0)\to (\C,0)$ denote the analytic family of functions of
Brian\c con-Speder's example (see Example \ref{BSexample}). We recall that $f_t$ is weighted homogeneous
of degree $15$ with respect to $w=(1,2,3)$, for all $t$.
When $t\neq 0$, equality holds in (\ref{interessant}) and thus
we observe that inequality (\ref{comparing}) is sharp. However $\LL_0(\nabla f_0)=\frac{d-w_0}{w_0}$
but the equality does not hold in (\ref{interessant}).

We also remark that the Brian\c con-Speder's example also shows that if $f:(\C^n,0)\to (\C,0)$ is a weighted homogeneous function
of degree $d$ with respect to a given vector of weights $w=(w_1,\dots, w_n)\in\Z_{\geq 1}$, then we can not
expect a formula for the whole sequence $\mu^*(f)$ in terms of $w$ and $d$.
\end{rem}

\begin{cor}
Let $(R,\m)$ be a quasi-unmixed Noetherian local ring of dimension $n$.
Let $I_1,\dots, I_n$ and $J_1,\dots, J_n$ be two families of ideals of $R$ of finite colength. Then
\begin{equation}\label{fitamult1}
\frac{e(I_1,\dots, I_n)}{e(J_1,\dots, J_n)}\leq
\LL_{J_1}(I_1)\LL_{J_2}(I_2)\cdots \LL_{J_n}(I_n).
\end{equation}
In particular, if $I$ is an ideal of $R$ of finite colength, then
\begin{equation}\label{fitamult2}
e(I)\leq \LL_0(I)^n.
\end{equation}
\end{cor}

\begin{proof}
Relation (\ref{fitamult1}) follows immediately as a recursive application of Proposition \ref{meollo}.
Inequality (\ref{fitamult2}) is a consequence of applying (\ref{fitamult1}) by considering
$I_1=\cdots=I_n=I$ and $J_1=\cdots=J_n=\m$.
\end{proof}

\begin{lem}\label{passquot}
Let $(R,\m)$ denote a Noetherian local ring of dimension $n$.
Let $I_1,\dots, I_n$ be ideals of $R$ such that $\sigma(I_1,\dots, I_n)<\infty$. Let $g\in I_n$ such that
$\dim R/\langle g\rangle=n-1$ and let $p:R\to R/\langle g\rangle$ denote the canonical projection. Then
$$
\sigma(I_1,\dots, I_n)\leq \sigma(p(I_1),\dots, p(I_{n-1})).
$$
\end{lem}

\begin{proof}
By Proposition \ref{sigmaexists}, there exist $g_i\in I_i$, for $i=1,\dots,n-1$, such that
$$
\sigma(p(I_1),\dots, p(I_{n-1}))=\sigma(p(g_1),\dots, p(g_{n-1})).
$$
The image in a quotient of $R$ of a given ideal of $R$ has multiplicity greater than or
equal to the multiplicity of the given ideal (see for instance \cite[Lemma 11.1.7]{HS} or \cite[p.\,146]{HIO}).
Therefore
$$
\sigma(p(I_1),\dots, p(I_{n-1}))=e(p(g_1),\dots, p(g_{n-1}))\geq e(g_1,\dots, g_{n-1},g)\geq \sigma(I_1,\dots, I_n)
$$
where the last inequality is a consequence of Lemma \ref{fitemlesmix}.
\end{proof}

\begin{prop}\label{femquocient}
Let $(R,\m)$ be a Noetherian local ring of dimension $n\geq 2$.
Let $J$ be a proper ideal of $R$ and let $I_1,\dots, I_n$ be ideals of $R$ such that $\sigma(I_1,\dots, I_n)<\infty$.
Let $g$ denote a sufficiently general element of $I_n$ and let $p:R\to R/\langle g\rangle$ denote the
canonical projection. Then
\begin{align}
\sigma(p(I_1),\dots, p(I_{n-1}))&=\sigma(I_1,\dots, I_n)\label{effectsuperf}\\
\LL_{p(J)}(p(I_1),\dots, p(I_{n-1}))&\leq \LL_J(I_1,\dots, I_n)\label{quocients}.
\end{align}
\end{prop}

\begin{proof}
Let us suppose that $g\in I_n$ is a superficial element for $I_1,\dots, I_n$ according to
\cite[Definition 17.2.1]{HS}. In particular, the element $g$ can be considered as a sufficiently general
element of $I_n$, by \cite[Proposition 17.2.2]{HS}. Therefore equality (\ref{effectsuperf}) holds,
by a result of Risler and Teissier \cite[Theorem 17.4.6]{HS} (see also \cite[p.\,306]{Cargese}).
From (\ref{effectsuperf}) we obtain the following chain of inequalities, for any pair of integers $r,s\geq 1$:
\begin{align}
\sigma(I_1^s,\dots, I_n^s)&=s^{n}\sigma(I_1,\dots, I_n)=s^{n}\sigma(p(I_1),\dots, p(I_{n-1}))\nonumber\\
&=s\cdot \sigma(p(I_1)^s,\dots, p(I_{n-1})^s)\geq s\cdot \sigma(p(I_1)^s+p(J)^r,\dots, p(I_{n-1})^s+p(J)^r)\nonumber\\
&\geq s\cdot \sigma(I_1^s+J^r,\dots, I_{n-1}^s+J^r, I_n)=\sigma(I_1^s+J^r,\dots, I_{n-1}^s+J^r, I_n^s)\label{lemmaquot}\\
&\geq \sigma(I_1^s+J^r,\dots, I_{n-1}^s+J^r, I_n^s+J^r),\nonumber
\end{align}
where the inequality of (\ref{lemmaquot}) is a direct application of Lemma \ref{passquot}. In particular,
we find that $r_{p(J)}(p(I_1)^s,\dots, p(I_{n-1})^s)\leq r_J(I_1^s,\dots, I_n^s)$, for all $s\geq 1$,
and hence relation (\ref{quocients}) follows.
\end{proof}

The next result shows an inequality that in some situations (see Corollary \ref{motivacio}) is more subtle than
inequality (\ref{fitamult1}). Moreover, Theorem \ref{gralHickel} constitutes
a generalization of the inequality proven by Hickel in \cite[Th\'eor\`eme 1.1]{Hickel}.

\begin{thm}\label{gralHickel} Let us suppose that
$(R,\m)$ is a quasi-unmixed Noetherian local ring. Let $I_1,\dots, I_n$ and $J_1,\dots, J_n$ two families of ideals of $R$ of finite colength. Then
\begin{align*}
\frac{e(I_1,\dots, I_n)}{e(J_1,\dots, J_n)}\leq & \,\,
\LL_{J_1}(I_1,J_2\dots, J_n)
\LL_{J_2}(I_2,I_2,J_3\dots, J_n)
\LL_{J_3}(I_3,I_3,I_3,J_4\dots, J_n)\\
&\cdots
\LL_{J_{n-1}}(I_{n-1},\dots,I_{n-1},J_n)
\LL_{J_n}(I_n,\dots,I_n).
\end{align*}
\end{thm}

\begin{proof} By Proposition \ref{meollo}, we have
\begin{equation}\label{1a}
e(I_1,\dots, I_n)\leq e(I_1,\dots, I_{n-1},J_n)\LL_{J_n}(I_n).
\end{equation}

Let $g_n\in J_n$ such that $\dim R/\langle g_n\rangle=n-1$ and let $p:R\to R/\langle g_n\rangle$ be the natural projection.
Therefore we obtain
\begin{equation}\label{2a}
e(I_1,\dots, I_{n-1},J_n)\leq e(p(I_1),\dots, p(I_{n-1})),
\end{equation}
by Lemma \ref{passquot}.
Applying again Proposition \ref{meollo} we have
\begin{align}
e(p(I_1),\dots, p(I_{n-1}))&\leq e(p(I_1),\dots, p(I_{n-2}), p(J_{n-1})) \LL_{p(J_{n-1})}(p(I_{n-1}))\nonumber\\
&\leq  e(p(I_1),\dots, p(I_{n-2}), p(J_{n-1})) \LL_{J_{n-1}}(I_{n-1},\dots, I_{n-1},J_n)\label{3a},
\end{align}
where (\ref{3a}) follows from Proposition \ref{femquocient}.
Thus joining (\ref{1a}), (\ref{2a}) and (\ref{3a}) we obtain
$$
e(I_1,\dots, I_n)\leq e\left(p(I_1),\dots, p(I_{n-2}), p(J_{n-1})\right) \LL_{J_{n-1}}(I_{n-1},\dots, I_{n-1},J_n)\LL_{J_n}(I_n).
$$
Now we can bound the multiplicity $e(p(I_1),\dots, p(I_{n-2}), p(J_{n-1}))$ by applying the same argument.
Then, by finite induction we construct a sequence of elements $g_i\in J_i$, for $i=2,\dots, n$, such that
$\dim R/\langle g_i,\dots, g_n\rangle=i-1$, for all $i=2,\dots, n$, and if $q$ denotes the projection $R\to R/\langle g_2,\dots, g_n\rangle$,
then
\begin{align*}
e(I_1,\dots, I_n)\leq & \,\,
e(q(I_1))
\LL_{J_2}(I_2,I_2,J_3\dots, J_n)
\LL_{J_3}(I_3,I_3,I_3,J_4\dots, J_n)\\
&\cdots
\LL_{J_{n-1}}(I_{n-1},\dots,I_{n-1},J_n)
\LL_{J_n}(I_n,\dots,I_n).
\end{align*}
By Propositions \ref{meollo} and \ref{femquocient} we have
$$
e(q(I_1))\leq e(q(J_1))\LL_{q(J_1)}(q(I_1))\leq e(q(J_1))\LL_{J_1}(I_1, J_2,\dots, J_n).
$$
Moreover, we can assume from the beginning that $g_n,g_{n-1},\dots, g_2$ forms a superficial sequence for
$J_n, J_{n-1},\dots, J_2, J_1$, in the sense of \cite[Definition 17.2.1]{HS}.
In particular we have the equality $e(q(J_1))=e(J_1,\dots, J_n)$, by \cite[Theorem 17.4.6]{HS}.
Thus the result follows.
\end{proof}

\begin{cor}\label{motivacio}
Let $(R,\m)$ be a quasi-unmixed Noetherian local ring and let $I$ and $J$ be ideals of $R$ of finite colength. Then
$$
\frac{e(I)}{e(J)}\leq \LL^{(1)}_J(I)\cdots \LL^{(n)}_J(I),
$$
where $
\LL_J^{(i)}(I)=\LL_J(\underbrace{I,\dots, I}_{\textnormal{$i$ times}},\underbrace{J,\dots, J}_{\textnormal{$n-i$ times}})$,
for $i=1,\dots, n$.
\end{cor}

\begin{proof}
It follows by considering $I_1=\cdots=I_n=I$ and $J_1=\cdots=J_n=J$ in the previous theorem.
\end{proof}

From the above result we conclude that if $f\in\O_n$ has an isolated singularity at the origin, then
$$
\mu(f)\leq \LL_0^{(1)}(\nabla f)\cdots \LL_0^{(n)}(\nabla f).
$$

We remark that Theorem \ref{gralHickel} and Corollary \ref{motivacio} are suggested by \cite[Remarque 4.3]{Hickel}.
Moreover, let us observe that the numbers $\nu_I^{(i)}$ defined by Hickel in \cite[p.\ 635]{Hickel} in a regular local ring
coincide with the numbers $\LL_0^{(i)}(I)$ introduced in Definition \ref{relativeLE}, as is shown in the following lemma.

\begin{lem}\label{genericvalue}
Let $(R,m)$ be a regular local ring with infinite residue field $k$.
Let $I$ be an ideal of $R$ of finite colength and let $i\in\{1,\dots, n-1\}$.
Then $\LL_0^{(i)}(I)$ is equal to the \L ojasiewicz exponent of the image of $I$ in the quotient ring
$R/\langle h_1,\dots, h_{n-i}\rangle$, where $h_1,\dots, h_{n-i}$ are linear forms chosen generically in $k[x_1,\dots, x_n]$ and
$x_1,\dots, x_n$ denote a regular parameter system of $R$.
\end{lem}

\begin{proof}
By \cite[Proposition 17.2.2]{HS} and \cite[Theorem 17.4.6]{HS}, we can take
generic lineal forms $h_1,\dots, h_{n-i}\in k[x_1,\dots, x_n]$ in order to have
$e(IR_H)=e_i(I)$, where $R_H$ denotes the quotient ring $R/\langle h_1,\dots, h_{n-i}\rangle$.
Let us denote by $\m_H$ the maximal ideal of $R_H$.
By \cite[Th\'eor\`eme 1.1]{Hickel}, the number $\LL_0(IR_H)$ does not depend on $h_1,\dots, h_{n-i}$.
Let us denote the resulting number by $\nu^{i}_I$, as in \cite{Hickel}.
We observe that
\begin{align*}
\LL_0(IR_H)&=\inf\left\{\frac rs:  \m_H^r\subseteq \overline{I^sR_H},\,\,\, r,s\geq 1  \right\}\\
&=\inf\left\{\frac rs:  e(I^sR_H)=e(I^sR_H+\m_H^r),\,\,\, r,s\geq 1\right\}.
\end{align*}
Moreover
$$
\LL_0^{(i)}(I)=\inf\left\{\frac rs:  e_i(I^s)=e_i(I^s+\m^r),\,\,\, r,s\geq 1\right\}.
$$
Let $r,s\geq 1$, then we have the following:
$$
e_i(I^s)=s^ie_i(I)=s^ie(IR_H)=e(I^sR_H)\geq e(I^sR_H+\m_H^r)\geq e_i(I^s+\m^r),
$$
where the last inequality follows from Lemma \ref{passquot}. In particular, if
$e_i(I^s)=e_i(I^s+\m^r)$, then $e(I^sR_H)=e(I^sR_H+\m_H^r)$. This means that $\LL_0(IR_H)\leq \LL_0^{(i)}(I)$ and consequently
$\nu^{i}_I\leq \LL_0^{(i)}(I)$.

Let us suppose that $\nu^{i}_I<\LL_0^{(i)}(I)$. Let $r,s\geq 1$ such that $\nu^{i}_I<\frac rs<\LL_0^{(i)}(I)$.
Therefore $e_i(I^s)>e(I^s+\m^r)$. Let us consider
generic linear forms $h_1,\dots, h_{n-i}\in k[x_1,\dots, x_n]$ such that $e_i(I^s)=e(I^sR_H)$ and $e_i(I^s+\m^r)=e((I^s+\m^r)R_H)$,
where $R_H=R/\langle h_1,\dots, h_{n-i}\rangle$. Since $\nu^{(i)}_I=\LL_0(IR_H)<\frac rs$, then $e(I^sR_H)=e((I^s+\m^r)R_H)$ and hence
$e_i(I^s)=e_i(I^s+\m^r)$, which is a contradiction. Therefore $\LL_0^{(i)}(I)=\nu^{i}_I$.
\end{proof}

\begin{lem}\label{chain}
Let $(R,\m)$ be a quasi-unmixed Noetherian local ring and let $I, J$ be ideals of $R$ of finite colength such that $I\subseteq J$. Let us suppose that
the residue field $k=R/\m$ is infinite. Let $i\in\{1,\dots, n-1\}$. If $e_{i+1}(I)=e_{i+1}(J)$, then $e_i(I)=e_i(J)$.
\end{lem}

\begin{proof}
Let $h_1,\dots, h_{n-i}\in \m$ sufficiently general elements of $\m$. Let us define
$R_1=R/\langle h_1,\dots, h_{n-i}\rangle$ and $R_2=\langle h_1,\dots, h_{n-i-1}\rangle$.
If $p:R\to R_1$ and $q:R\to R_2$ denote the natural projections, then
$e_{i}(I)=e(p(I)R_1)$, $e_i(J)=e(p(J)R_1)$, $e_{i+1}(I)=e(q(I)R_2)$ and $e_{i+1}(J)=e(q(J)R_2)$.
Since the ring $R_2$ is also quasi-unmixed (see for instance \cite[Proposition B.44]{HS}), the condition
$e_{i+1}(I)=e_{i+1}(J)$ implies that $\overline{q(I)}=\overline{q(J)}$,
where the bar denotes integral closure in $R_2$, by the Rees' multiplicity theorem.
In particular we have $\overline{p(I)}=\overline{p(J)}$, as an equality of integral closures in $R_1$. Thus
$e(p(I)R_1)=e(p(J)R_1)$ and the result follows.
\end{proof}

\begin{cor}
Let $(R,\m)$ be a quasi-unmixed Noetherian local ring and let $I, J$ be ideals of $R$ of finite colength. Let us suppose that
the residue field $k=R/\m$ is infinite.
Then $\LL_J^{(1)}(I)\leq\cdots \leq \LL_J^{(n)}(I)$.
\end{cor}

\begin{proof}
Let us fix an index $i\in\{1,\dots, n-1\}$. Let us fix two integers $r,s\geq 1$ such that $e_{i+1}(I^s)=e_{i+1}(I^s+J^r)$.
Then $e_{i}(I^s)=e_{i}(I^s+J^r)$, by Lemma \ref{chain}. Hence the result follows from the definition of $\LL_J^{(i)}(I)$.
\end{proof}

\lsection{Mixed \L ojasiewicz exponents of monomial ideals}\label{MLojmonomial}

Let $v\in\R^n_{\geqslant 0}$, $v=(v_1,\dots, v_n)$. We define
$v_{\min}=\min\{v_1,\dots, v_n\}$ and
$A(v)=\{j: v_j=v_{\min}\}$.
Given an index $i\in\{1,\dots, n\}$, we define
$S^{(i)}=\{v\in\R_{>0}^n:\# A(v)\geq n+1-i\}$ and
$S^{(i)}_0=\{v\in\R_{> 0}^n:\# A(v)= n+1-i\}$. We observe that
$S^{(1)}=S^{(1)}_0=\{(\lambda, \dots, \lambda): \lambda>0\}$, $S^{(n)}=\R^n_{>0}$
and $S^{(i)}_0=S^{(i)}\smallsetminus S^{(i-1)}$, for all
$i=1,\dots, n$, where we set $S^{(0)}=\emptyset$.

If $h\in\O_n$ and $h=\sum_ka_kx^k$ denotes the Taylor expansion of $h$ around the origin, then
{\it support} of $h$ is defined as the set $\supp(h)=\{k\in\Z^n_{\geq 0}: a_k\neq 0\}$. If $h\neq 0$, the {\it Newton polyhedron}
of $h$, denoted by $\Gamma_+(h)$, is the convex hull in $\R^n$ of the set
$\{k+v: k\in\supp(h),\,v\in\R^n_{\geq 0}\}$. If $h=0$, then we set $\Gamma_+(h)=\emptyset$.
If $I$ denotes an ideal of $\O_n$ and $g_1,\dots, g_r$ is a generating system of $I$, then the
{\it Newton polyhedron of $I$}, denoted by $\Gamma_+(I)$, is defined as the convex hull of
$\Gamma_+(g_1)\cup\cdots \cup \Gamma_+(g_r)$. It is easy to check that the definition of $\Gamma_+(I)$
does not depend on the chosen generating system $g_1,\dots, g_r$ of $I$.

If $v\in\R_{\geq 0}^n$ and $I$ denotes an ideal of $\O_n$, then we define
$$
\ell(v, I)=\min\left\{\langle v,k\rangle: k\in\Gamma_+(I)\right\},
$$
where $\langle\, ,\rangle$ stands for the standard scalar product in $\R^n$.
Therefore, if $v=(1,\dots, 1)\in\R^n_{\geq 0}$, then
$\ell(v,I)=\ord(I)$, where $\ord(I)$ is the {\it order of $I$}, that is, the minimum of those $r\geq 1$ such that $I\subseteq \m^r$.
If $h\in \O_n$ and $v\in\R^n_{>0}$, then the number $\ell(v, h)$ is also denoted by $d_v(h)$ and we refer to $d_v(h)$ as the
{\it degree of $h$ with respect to $v$}.

\begin{thm}\label{Lojmonomial}
If $I$ is a monomial ideal of $\O_n$ of finite colength, then
$$
\LL^{(i)}_0(I)=\max\left\{\frac{\ell(v, I)}{v_{\min}}:v\in
S^{(i)}\right\},
$$
for all $i=1,\dots, n$.
\end{thm}

\begin{proof}
Let us fix an index $i\in\{1,\dots, n\}$.
The closures of connected components $S_0^{(i)}$ form a regular subdivision
corresponding to the blow up at the origin.
Let us consider a regular subdivision $\Sigma$ of the dual Newton
 polyhedron of $\Gamma_+(I)$, which is
 also a subdivision of $\{S_0^{(i)}\}$.
Then we have a natural map from $\Sigma$ to $\{S_0^{(i)}\}$.
Take a vector $a$ which is a
generator of 1-cone of $\Sigma$ and denote by $E_a$ the corresponding
exceptional divisor. Then $E_a$ meets $L'$ if and only if
the cone generated by $a$ is in the closure of some connected component of
$S_0^{(i)}$, $i\geq n+1-k$,
where $L'$ denotes the strict transform of $L$.
So \eqref{L_X} implies the result.
\end{proof}


Let us fix a subset $L\subseteq\{1,\dots,n\}$, $L\neq\emptyset$.
Then we define $\R^n_L=\{x\in\R^n: x_i=0,\,\textnormal{for all $i\notin L$}\}$.
If $h\in\O_n$ and $h=\sum_ka_kx^k$ is the Taylor expansion of $h$ around the origin,
then we denote by $h_L$ the sum of all terms $a_kx^k$ such that $k\in\R^n_L$; if no such terms exist
then we set $h_L=0$. Let $\O_{n,L}$ denote the subring of $\O_n$ formed by all function germs of $\O_n$ that only depend on
the variables $x_i$ such that $i\in L$. If $I$ is an ideal of $\O_n$, then $I^L$ denotes the ideal of $\O_{n,L}$
generated by all $h_L$ such that $h\in I$. In particular, if $I$ is an ideal of $\O_n$ of finite colength
then $I^{\{i\}}\neq 0$, for all $i=1,\dots, n$.

\begin{cor}\label{Lojorder} Let $I$ be a monomial ideal of $\O_n$ of finite colength. Then, for all $i\in\{1,\dots, n\}$, we have
\begin{equation}\label{eqLojorder}
\LL_0^{(i)}(I)=\max\big\{\ord(I^{\{j_1,\dots, j_{n+1-i}\}}): 1\leqslant j_1<\cdots<j_{n+1-i}\leqslant n\big\}.
\end{equation}
\end{cor}

\begin{proof}
Let us fix an index $i\in\{1,\dots, n\}$ and let us denote
the number on the right
hand side of (\ref{eqLojorder}) by $m_i(I)$.
If $v\in\R^n_{>0}$, then we denote the vector $\frac{1}{v_{\min}}v$ by $w_v$.
If $w_v=(w_1,\dots, w_n)$, then we observe that $w_j=1$ whenever $j\in A(v)$ and $w_j>1$, otherwise.

By Theorem \ref{Lojmonomial} we have
$$
\LL^{(i)}_0(I)=\max\left\{\ell(w_v, I): v\in S^{(i)}\right\}.
$$

We remark that, since $I$ is an ideal of finite colength, then $I^L\neq 0$, for all
$L\subseteq \{1,\dots, n\}$, $L\neq\emptyset$.
Let us fix a vector $v\in S^{(i)}$.
Then from the inclusion $I^{A(v)}\subseteq I$ we deduce $\ell(w_v, I)\leq \ell(w_v, I^{A(v)})=\ord(I^{A(v)})$.
In particular, we have
\begin{align*}
\hspace{1cm}\max\left\{\ell(w_v, I): v\in S^{(i)}\right\}&\leq \max\left\{\ell(w_v, I^{A(v)}): v\in S^{(i)}\right\}\nonumber\\
&=\max\left\{\ord(I^{A(v)}): v\in S^{(i)}\right\}\nonumber\\
&\leq\max\left\{\ord(I^{A(v)}): v\in S_0^{(i)}\right\}\label{ords}\\
&=\max\left\{\ord(I^{\{j_1,\dots, j_{n+1-i}\}}): 1\leq j_1<\cdots j_{n+1-i}\leq n\right\}\nonumber.
\end{align*}
Hence $\LL_0^{(i)}(I)\leq m_i(I)$. Let us see the converse inequality by proving that for any subset
$L\subseteq\{1,\dots, n\}$ such that $\vert L\vert=n+1-i$, there exist some vector $v\in \R^n_{>0}$
such that $A(v)=L$ and $\ell(w_v, I)=\ord(I^L)$.

Let us fix a subset $L\subseteq\{1,\dots, n\}$ such that $\vert L\vert=n+1-i$ and let $v=(v_1,\dots, v_n)\in \R^n$ such that
$v_i=1$ for all $i\in L$ and $v_j>\ord(I^L)$, for all $j\notin L$. Let us observe that, if $x^k\notin I^L$, then there exists some
$j_0\notin L$ such that $k_{j_0}\geq 1$; in particular $\langle v,k\rangle\geq \ord(I^L)$.
Therefore we have
$$
\ell(w_v, I)=\ell(v, I)=\min\left\{ \min_{x^k\in I^L} \langle k,v\rangle,\,\min_{x^k\notin I^L} \langle k,v\rangle \right\}=
\min\left\{ \ord(I^L),\,\min_{x^k\notin I^L} \langle k,v\rangle \right\}=\ord(I^L).
$$
Thus the result follows.
\end{proof}

\begin{rem}
If $I$ denotes an ideal of finite colength of $\O_n$ then we observe that $\LL^*_0(I)=\LL_0^*(\overline I)$.
Therefore in Theorem \ref{Lojmonomial} and Corollary \ref{Lojorder} we can replace the ideal $I$ by any ideal
of $\O_n$ whose integral closure is a monomial ideal.
\end{rem}

\begin{ex}
Let us consider the monomial ideal of $\O_3$ given by
$I=\langle x^a,y^b,z^c,xyz\rangle$, where $a,b,c\in\Z_{\geq 0}$ and $3<a<b<c$.
Using the formula $e(I)=3!\V_n(\R_{\geq 0}^3\smallsetminus \Gamma_+(I))$ we obtain $e(I)=ab+ac+bc$.
Moreover $\LL^*_0(I)=(c,b,3)$, by Corollary \ref{Lojorder}. We remark that $\LL^*_0(I)$ does not depend on
$a$.
\end{ex}

\begin{ex}\label{BSexample}
Let us consider the family $f_t:(\C^3,0)\to (\C,0)$ given by:
$$
f_t(x,y,z)=x^{15}+z^5+xy^7+ty^6z.
$$
This is known as the Brian\c{c}on-Speder's example (see \cite{BS}).
We have that $f_t$ has an isolated singularity at the origin, $f_t$
is weighted homogeneous with respect to $w=(1,2,3)$ and $d_w(f_t)=15$, for all $t$. Therefore
$\LL_0(\nabla f_t)=14$, for all $t$, by \cite{KOP}. It is known that $\mu^{(2)}(f_0)=28$
and $\mu^{(2)}(f_t)=26$, for all $t\neq 0$ (see \cite{BS}). Hence
$$
\mu^*(f)=
\begin{cases}
(364, 28, 5)  &\textrm{if $t=0$}\\
(364, 26, 5)  &\textrm{if $t\neq 0$}.
\end{cases}
$$

It is straightforward to check that the ideal $J(f_0)$ is Newton non-degenerate, in the sense of \cite[p.\,57]{BFS}.
Thus the integral closure of $J(f_0)$ is a monomial ideal. That is
$$
\overline{J(f_0)}=\overline{\langle x^{14}, y^7, xy^6, z^4\rangle}.
$$
In particular, we can apply Corollary \ref{Lojorder} to deduce
$$
\LL_0^*(\nabla f_0)=(14,7,5).
$$

If $t\neq 0$, then $\Gamma_+(J(f_t))=\Gamma_+(J)$, where $J$ is the monomial ideal given by
$J=\langle x^{14}, y^6, z^4, y^5z, xy^6\rangle$. Obviously $J\subseteq J(f_t)$. We observe that $e(J)=336$, whereas $e(J(f_t))=364$. Since
$e(J)\neq e(J(f_t))$ we conclude that the ideal $J(f_t)$ is not Newton non-degenerate.
In particular, we can not apply Corollary \ref{Lojorder} to obtain the sequence $\LL_0^*(\nabla f_t)$.

Let us compute the number $\LL_0^{(2)}(J(f_t))$, for $t\neq 0$. Let us fix a parameter $t\neq 0$.
We remark that $\LL_0^{(2)}(J(f_t))$ is equal to the \L ojasiewicz exponent
of the function $g(x,y)=f_t(x,y,ax+by)$, for generic values $a,b\in \C$, by Lemma \ref{genericvalue} and
\cite[Proposition 2.7]{Cargese}.

We recall that if $I$ denotes an ideal of $\O_n$ of finite colength, then we denote by $r_0(I)$
the minimum of those $r\geq 1$ such that $\m^r\subseteq \overline I$.
Using Singular \cite{Singular} we observe that $r_0(J(g))=7$.

By a result of P\l oski \cite[Proposition 3.1]{Ploski}, it is enough to compute the quotients $\frac{r_0(J(g)^s)}{s}$
only for those integers $s$ such that $1\leq s\leq r_0(J(g)^s)\leq e(J(g))=26$. Moreover, since $r_0(J(g))-1<\LL_0(J(g))=
\inf_{s\geq 1}\frac{r_0(J(g)^s)}{s}$, we can consider only the integers $s$ such that
$1\leq s\leq \frac{e(J(g))}{r_0(J(g))-1}=\frac{26}{6}\simeq 4.3$, that is, such that $1\leq s\leq 4$.
Again, by applying Singular \cite{Singular} we obtain
$$
r_0(J(g))=7  \hspace{1cm}    r_0(J(g)^2)=13  \hspace{1cm}    r_0(J(g)^3)=20   \hspace{1cm}    r_0(J(g)^4)=26.
$$
Then $$\LL_0(J(g))=\min\left\{\frac{r_0(J(g))}{1},\frac{r_0(J(g)^2)}{2},\frac{r_0(J(g)^3)}{3},\frac{r_0(J(g)^4)}{4}\right\}=6.5.$$
Summing up the above information we conclude
$$
\LL^*_0(\nabla f_t)=
\begin{cases}
(14,7,5)  &\textrm{if $t=0$}\\
(14,6.5,5)  &\textrm{if $t\neq 0$}.
\end{cases}
$$

It is known that the deformation $f_t:(\C^3,0)\to (\C,0)$ is topologically trivial (see \cite{BS}).
However, this deformation is not bi-Lipschitz $\mathcal R$-trivial, as is observed by Koike \cite{Koike}.
Therefore, the fact that $\LL^*_0(\nabla f_0)\neq \LL^*_0(\nabla f_t)$, for $t\neq 0$, in this example constitutes a clue
pointing that, if $f\in\O_n$ is a function germ having an isolated singularity at the origin, then
the sequence $\LL^*_0(\nabla f)$ might be invariant in the bi-Lipschitz $\mathcal R$-orbit of $f$.
\end{ex}

\lsection{The bi-Lipschitz invariance of the \L ojasiewicz exponent}
\label{biLipLoj}
In this section we show three theorems.
The first one shows that $\LL_0(\nabla f)$ is
bi-Lipschitz $\mathcal A$-invariant and bi-Lipschitz $\mathcal K^*$-invariant, for any $f\in\O_n$ with an isolated singularity at the origin.
The second shows the bi-Lipschitz invariance of $\LL_0(I)$ and $\ord(I)$, for
any ideal $I$ of $\O_n$ of finite colength. The third one concerns the invariance of $\LL_0(\nabla f)$ in
$\mu$-constant deformations of $f$.

\begin{thm}\label{BLInvFunc}
Let $f,g\in\O_n$ with an isolated singularity at the origin.
Let us suppose that $f$ and $g$ are
bi-Lipschitz $\mathcal A$-equivalent or
bi-Lipschitz $\mathcal K^*$-equivalent. Then $\LL_0(\nabla f)=\LL_0(\nabla g)$.
\end{thm}

\begin{proof}
By symmetry, it is enough to show $\LL_0(\nabla f)\leq \LL_0(\nabla g)$.
Let us consider a bi-Lipschitz homeomorphism
$\varphi:(\C^n,0)\to(\C^n,0)$ and a
bi-Lipschitz homeomorphism $\phi:(\C,0)\to(\C,0)$
so that $g(\varphi(x))=\phi(f(x))$,
for all $x$ belonging to some open neighbourhood of $0\in\C^n$.
By Rademacher's theorem
(see for instance \cite[Theorem 5.1.11]{GIT}),
the partial derivatives of $\varphi$ and $\varphi^{-1}$
exist in some open neighbourhood of $0\in\C^n$ except in a thin set. The bi-Lipschitz property
implies that $\varphi$ and $\varphi^{-1}$ are bounded. Then we conclude that
\begin{align}
\|\nabla g(\varphi(x))\|\lesssim
\|\nabla g(\varphi(x))D\varphi(x)\|
=\|D\phi(f(x))\nabla f(x)\|\lesssim\|\nabla f(x)\|
\end{align}
almost everywhere.
By continuity, we have $\|\nabla g(\varphi(x))\|\lesssim\|\nabla f(x)\|$ near $0$.
If $\|x\|^\theta\lesssim\|\nabla g(x)\|$, then
$$
\|x\|^\theta
\sim\|\varphi(x)\|^\theta
\lesssim\|\nabla g(\varphi(x))\|
\lesssim\|\nabla f(x)\|
$$
and we obtain $\LL_0(\nabla f)\leq\LL_0(\nabla g)$.

The proof for $\mathcal K^*$-equivalence is similar.
Let $A:(\C^n,0)\to \C^*$ be a Lipschitz map such that the map $A^{-1}:(\C^n,0)\to \C^*$ defined by
$A^{-1}(x)=A(x)^{-1}$ is Lipschitz and $g(\varphi(x))=A(x)f(x)$, for all $x$ belonging
to some open neighbourhood of the origin. Then we obtain that
\begin{align*}
\|\nabla g(\varphi(x))\|
\lesssim&\|\nabla g(\varphi(x))D\varphi(x)\|
&&(\textrm{since $\varphi^{-1}$ is Lipschitz})\\
=&\|\nabla A(x)f(x)+A(x)\nabla f(x)\|
&&(\textrm{since }g(\varphi(x))=A(x)f(x))\\
\le&\|\nabla A(x)\||f(x)|+|A(x)|\|\nabla f(x)\|\\
\lesssim&|f(x)|+\|\nabla f(x)\|
&&(\textrm{since $A(x)$ is Lipschitz})\\
\lesssim&\|x\|\|\nabla f(x)\|+\|\nabla f(x)\|
&&(\textrm{since }|f(x)|\lesssim\|x\|\|\nabla f(x)\|)\\
\lesssim&\|\nabla f(x)\|,
\end{align*}
almost everywhere and we conclude that $\LL_0(\nabla f)\leq\LL_0(\nabla g)$.
\end{proof}

\begin{thm}\label{BLinv}
Let $I$ and $J$ be ideals of $\O_n$ such that $I$ and $J$ are bi-Lipschitz equivalent.
Then $\ord(I)=\ord(J)$, and $\LL_0(I)=\LL_0(J)$ if $I$ and $J$ have finite colength.
\end{thm}

\begin{proof}
Since $I$ and $J$ are bi-Lipschitz equivalent, there exist analytic map germs
$f=(f_1,\dots,f_p):(\C^n,0)\to (\C^p,0)$ and $g=(g_1,\dots,g_q): (\C^n,0)\to (\C^q,0)$ such that
$\overline{I}=\overline{\langle f_1,\dots,f_p\rangle}$ ,
$\overline{J}=\overline{\langle g_1,\dots,g_q\rangle}$
and there exists a bi-Lipschitz homeomorphism
 $\varphi:(\C^n,0)\to(\C^n,0)$ so that
$\|g(\varphi(x))\|\sim\|f(x)\|$ near $0$.
By symmetry, it is enough to show that $\LL_0(I)\leq\LL_0(J)$ and
$\ord(I)\leq\ord(J)$.

Let $\theta\in\R_{\geq 0}$ such that $\|x\|^\theta\lesssim\|g(x)\|$ near $0$. Then
$$
\|x\|^\theta\sim\|\varphi(x)\|^\theta\lesssim\|g(\varphi(x))\|\sim\|f(x)\|
$$
near $0$ and we obtain that $\LL_0(I)\leq\LL_0(J)$.

We remark that
$$
\ord(J)=\max\{s:J\subseteq\m_n^s\}
=\max\{s:J\subseteq\overline{\m_n^s}\}
=\max\{s:\|g(x)\|\lesssim\|x\|^s\,\,\textnormal{near $0$}\}.
$$
If $\|f(x)\|\lesssim\|x\|^s$ near $0$, then we have
$$
\|g(x)\|\sim\|f(\varphi(x))\|\lesssim
\|\varphi(x)\|^s\sim\|x\|^s
$$
near $0$ and we obtain $\ord(I)\leq \ord(J)$.
\end{proof}

To end this section we show a result about the constancy of $\LL_0(\nabla f_t)$ in deformations
of weighted homogeneous functions.

\begin{thm}\label{Lojgradconst}
Let $f:(\C^n,0)\to (\C,0)$ be a weighted homogeneous function of degree $d$ with respect
to $w=(w_1,\dots, w_n)$ with an isolated singularity at the origin. Let $w_0=\min\{w_1,\dots, w_n\}$.
Let us suppose that
\begin{equation}\label{igualtatperaL}
\LL_0(\nabla f)=\min\left\{\mu(f),\frac{d-w_0}{w_0}\right\}.
\end{equation}
Let $f_t:(\C^n,0)\to (\C,0)$ be an analytic deformation of $f$ such that $f_t$ has an isolated singularity at the origin, for
all $t$. If $\mu(f_t)$ is constant, then $\LL_0(\nabla f_t)$ is also constant.
\end{thm}

\begin{proof}
By a result of Varchenko \cite{Var} (see also \cite[Proposition 2]{OShea}), the deformation $f_t$ verifies $d_w(f_t)\geq d$, for all $t$, where
$d_w(f_t)$ denotes the degree of $f_t$ with respect to $w$.
Then we have the following:
$$
\frac{(d-w_1)\cdots (d-w_n)}{w_1\cdots w_n}=\mu(f)=\mu(f_t)\geq \frac{(d_t-w_1)\cdots (d_t-w_n)}{w_1\cdots w_n}\geq
\frac{(d-w_1)\cdots (d-w_n)}{w_1\cdots w_n}.
$$
Therefore $d_w(f_t)=d$ and
$$
\mu(f_t)=\frac{(d-w_1)\cdots (d-w_n)}{w_1\cdots w_n}
$$
for all $t$. Consequently $f_t$ is a semi-weighted homogeneous function, for all $t$, by \cite[Theorem 3.3]{BFS} (see also
\cite{FT}). Then, by \cite[Corollary 4.7]{BE2012}, we obtain
$$
\LL_0(\nabla f_t)\leq \frac{d-w_0}{w_0}.
$$
By the lower semi-continuity of \L ojasiewicz exponents in $\mu$-constant deformations
(see \cite{Ploski2}) we have
$$
\min\left\{\mu(f),\frac{d-w_0}{w_0}\right\}=\LL_0(\nabla f)\leq \LL_0(\nabla f_t)\leq \min\left\{\mu(f_t),\frac{d-w_0}{w_0}\right\}=
\min\left\{\mu(f),\frac{d-w_0}{w_0}\right\}.
$$
Then the result follows.
\end{proof}

Since the order of a function can be seen as a \L ojasiewicz exponent, that is $\ord(f)=\LL_{\langle f\rangle}(\m_n)$,
for all $f\in\m_n$, we can consider the previous result as a counterpart of the known results
of O'Shea \cite[p.\,260]{OShea} and Greuel \cite[p.\,164]{Greuel} in the context of \L ojasiewicz exponents of gradient maps.
We remark that in general we always have the inequality $(\leq)$ in (\ref{igualtatperaL}).

\lsection{Log canonical thresholds}\label{LogCanTh}

The purpose of this section is to show in Theorem \ref{lctBL}
that the log canonical threshold $\lct(I)$ is bi-Lipschitz invariant.
We also show Theorem \ref{LCTLojExp}, which enables us to compute $\lct(I)$
in terms of \L ojasiewicz exponents when $\overline I$ is monomial.
We start with a quick survey on log canonical thresholds.
We refer to the survey \cite{MustataIMPANGA} for more information about
the notion of log canonical threshold.

The {\it log canonical threshold} of a function $f:(\C^n,0)\to\C$,
denoted by $\lct(f)$, is the
supremum of those $s$ so that $|f(x)|^{-2s}$ is locally integrable at $0$,
that is, integrable on some compact neighbourhood of $0$.
This definiton is generalized for ideals as follows.
\begin{defn}
Let $I$ be an ideal of $\O_n$. Let us consider a generating system $\{g_1,\dots, g_r\}$
of $I$. The {\it log canonical threshold} of $I$,
denoted by $\lct(I)$, is defined as follows:
$$
\lct(I)=\sup\{s\in\R_{\geq 0}:
\bigl(|g_1(x)|^{2}+\cdots+|g_r(x)|^2\bigr)^{-s}\ \textrm{ is locally
 integrable at $0$}\}.
$$
It is straightforward to see that this definition does not depend on the
choice of generating systems of $I$.
The {\it Arnold index} of $I$,
denoted by $\mu(I)$, is defined as $\mu(I)=\frac{1}{\lct(I)}$
(see for instance \cite{dFEM1}).
\end{defn}

One origin of the notion of log canonical threshold
comes back to analysis on complex powers as
generalized functions. M.\,Atiyah (\cite{Atiyah})
showed a way to compute (candidate) poles of complex powers
using resolution of singularities.
This leads to the following well-known result.

\begin{thm}
Let $\pi:M\to\C^n$ be a proper modification so that
$(\pi^*I)_0=\sum_im_iD_i$ where $D_i$ form a family of normal crossing
divisors.
Then
$$
\lct(I)=\min_i\Bigl\{\frac{k_i+1}{m_i}\Bigr\}\qquad\textrm{
where $K_M=\sum_ik_iD_i$ is the canonical divisor of $M$.}
$$
\end{thm}
The proof is based on the following observation:
$$
\int_{\|x\|\le\varepsilon}
{|x_1^{m_1}\cdots x_n^{m_n}|^{-2s}}|x_1^{k_1}\cdots x_n^{k_n}|^2
\frac{dx\wedge d\bar{x}}{\sqrt{-1}^n}<\infty
\quad \Longleftrightarrow\quad
m_is<k_i+1, \textrm{\,for all $i$}.
$$

If $I\subseteq\m_n^r$, then
$$
\lct(I)\leq\lct(\m_n^r)\leq\frac{\lct(\m_n)}r=\frac{n}r
$$
by \cite[Property 1.14]{MustataIMPANGA}.
As a consequence, we conclude that $\lct(I)\ord(I)\leq n$.
Combining this with
{\cite[Property 1.18]{MustataIMPANGA}}, we have
$$
\frac1{\ord(I)}\leq\lct(I)\leq\frac{n}{\ord(I)}.
$$

\begin{thm}\label{lctBL}\
\begin{enumerate}
\setlength{\parskip}{0cm} %
\setlength{\itemsep}{0cm} %
\item[(i)]
If two functions $f$ and $g$ of $\O_n$ are bi-Lipschitz $\mathcal K$-equivalent,
then $\lct(f)=\lct(g)$.
\item[(ii)]
If two ideals $I$ and $J$ of $\O_n$
are bi-Lipschitz equivalent, then $\lct(I)=\lct(J)$.
\end{enumerate}
\end{thm}

\begin{proof}
(i): Assume that we have $g(\varphi(x))=\phi_x(f(x))$, for all $x$ belonging to some open neighbourhood of $0\in\C^n$, for
a bi-Lipschitz homeomorphism
$\varphi:(\C^n,0)\to(\C^n,0)$, $x\mapsto x'=\varphi(x)$,
and bi-Lipschitz homeomorphisms
$\phi_x:(\C,0)\to(\C,0)$, $y\mapsto y'=\phi_x(y)$.
By Rademacher's theorem (see \cite[Theorem 5.1.11]{GIT}),
$\varphi$ is differentiable almost everywhere in
the sense of Lebesgue measure, and its jacobian $J(\varphi)$ is measurable.
By Lipschitz property, we have $|J(\varphi)|\lesssim1$ and
$|\phi_x(y)|\sim|y|$. So we have
\begin{align*}
\int_{\varphi(K)}|g(x')|^{-2s}\frac{dx'\wedge d\bar{x}'}{\sqrt{-1}^n}
=&\int_{K}|g(\varphi(x))|^{-2s}|J(\varphi)|
\frac{dx\wedge d\bar{x}}{\sqrt{-1}^n}\\
\lesssim&\int_K|\phi_x(f(x))|^{-2s}\frac{dx\wedge d\bar{x}}{\sqrt{-1}^n}\\
\lesssim&\int_K|f(x)|^{-2s}\frac{dx\wedge d\bar{x}}{\sqrt{-1}^n}
\end{align*}
where $K$ is a compact neighbourhod of $0$.
This implies $\lct(f)\leq\lct(g)$ and vice versa. \\
(ii):
Choose $f=(f_1,\dots,f_p)$ and $g=(g_1,\dots,g_q)$
so that
$\overline{I}=\overline{\langle f_1,\dots,f_p\rangle}$,
$\overline{J}=\overline{\langle g_1,\dots,g_q\rangle}$ and
$\|f(x)\|\sim\|g(\varphi(x))\|$ where
$\varphi:(\C^n,0)\to(\C^n,0)$
is a bi-Lipschitz homeomorphism.
We have
$$
\int_{\varphi(K)}\|g(x')\|^{-2s}\frac{dx'\wedge d\bar{x}'}{\sqrt{-1}^n}
=
\int_K
\|g(\varphi(x))\|^{-2s}|J(\varphi)|\frac{dx\wedge d\bar{x}}{\sqrt{-1}^n}
\lesssim\int_K\|f(x)\|^{-2s}\frac{dx\wedge d\bar{x}}{\sqrt{-1}^n}
$$
where $K$ is a compact neighbourhod of $0$.
This implies $\lct(I)\leq\lct(J)$ and vice versa.
\end{proof}

\begin{thm}\label{LCTLojExp}
Let $I$ be an ideal of $\O_n$ such that $V(I)\subseteq V(x_1\cdots x_n)$.
We have
\begin{equation}\label{lctandL}
1\leq{\lct(I)}\mathcal L_{x_1\cdots x_n}(I)
\end{equation}
and equality holds when $\overline{I}$ is a monomial ideal.
\end{thm}

\begin{proof}
Let us consider an analytic map germ $f=(f_1,\dots, f_p):(\C^n,0)\to (\C^p,0)$ such that
$I=\langle f_1,\dots, f_p\rangle$.
Let $\theta\in\R_{\geq 0}$ such that $|x_1\dots x_n|^\theta\lesssim\|f(x)\|$. If $s\geq 0$ then
$$
\int_K\|f(x)\|^{-2s}\frac{dx\wedge d\bar{x}}{\sqrt{-1}^n}
\lesssim
\int_K|x_1\cdots x_n|^{-2s\theta}\frac{dx\wedge d\bar{x}}{\sqrt{-1}^n}.
$$
Thus $s<\lct(I)$ whenever $s\theta<1$.
This implies that $1/\mathcal L_{x_1\cdots x_n}(I)\leq \lct(I)$.

If $\overline{I}$ is monomial, then we consider the toric modification
$\pi:M\to\C^n$ corresponding to a regular subdivision of $\Gamma_+(I)$.
Let $a$ denote a primitive vector which generate a 1-cone of this
 regular fan.
Then the order of $|x_1\cdots x_n|^\theta\comp\pi$ is 
$\sum_{i=1}^n a_i\theta=(k_a+1)\theta$ along the
exceptional divisor corresponding to $a$, where
$k_a$ denotes the multiplicity of the canonical divisor along the component corresponding to $a$.
The order of $|f\comp\pi|$ is $\ell(a,I)$ along the
 exceptional divisor corresponding to $a$.
So we have
$$
\mathcal L_{x_1\cdots x_n}(I)
=\max\Bigl\{\frac{\ell(a,I)}{\sum_i{a_i}}
\Bigr\}
=\frac1{\lct(I)},
$$
where the maximum is taken over those $a$ which correspond to the
components of the exceptional divisor of $\pi$.
\end{proof}

The previous result is motivated by \cite[Example 5]{Howald}.

\begin{ex}
Let us consider the ideal $I=\langle x+y,xy\rangle$ of $\C[[x,y]]$.
Then $\mathcal L_{xy}(I)=1$ and $\lct(I)=3/2$.
We remark that $\overline{I}=\langle x+y\rangle+\langle x,y\rangle^2$. Hence this example shows that,
in general, equality does not hold in (\ref{lctandL}).
\end{ex}

\begin{prop}
Let $I$ and $J$ be ideals of $\O_n$ such that $V(J)\subseteq V(I)$. Then
\begin{equation}\label{lctIJ}
\lct(I)\leq\mathcal L_I(J)\lct(J).
\end{equation}
\end{prop}

\begin{proof}
Set $I=\langle f_1,\dots,f_r\rangle$ and
$J=\langle g_1,\dots,g_s\rangle$.
If $\|f(x)\|^\theta\lesssim\|g(x)\|$, for some $\theta\in\R_{\geq 0}$ and we fix any $s\geq 0$ then
$$
\int_K\|g(x)\|^{-2s}\frac{dx\wedge d\bar{x}}{\sqrt{-1}^n}
\lesssim
\int_K\|f(x)\|^{-2s\theta}\frac{dx\wedge d\bar{x}}{\sqrt{-1}^n}.
$$
This means that $s\theta<\lct(I)$ implies $s<\lct(J)$, i.e.,
$\lct(I)/\theta\leq\lct(J)$.
We thus obtain that $\lct(I)\le\theta\lct(J)$.
\end{proof}

\begin{rem}
It is natural to ask when the equality holds in \eqref{lctIJ}.
If $I$ and $J$ are monomial ideal of $\O_n$, then we have
$$
\lct(I)=\min_{a\in\R^n_{>0}}\Bigl\{\frac{\sum_ia_i}{\ell(a,I)}\Bigr\}, \quad
\lct(J)=\min_{a\in\R^n_{>0}}\Bigl\{\frac{\sum_ia_i}{\ell(a,J)}\Bigr\}, \quad
\mathcal L_I(J)=\max_{a\in\R^n_{>0}}\Bigl\{\frac{\ell(a,J)}{\ell(a,I)}\Bigr\}.
$$
When the same $a$ attains these minimums and maximum,
we have $\lct(I)=\mathcal L_I(J)\lct(J)$.
\end{rem}

\lsection{Log canonical thresholds of generic sections}\label{LogCanTh2}

\begin{defn}
Let $I$ be an ideal of $\O_n$. For any integer $k\in\{0,1,\dots, n-1\}$ we set
$$
\lct^{(n-k)}(I)=\lct(I|_L),
$$
where $L$ denote a generic $(n-k)$-dimensional linear subspace of
 $\C^n$,
and $I|_L$ denote the restriction of the ideal $I$ to the space $L$.
\end{defn}
By the semicontinuity of the log canonical threshold
(\cite[Corollary 9.5.39]{Laz}),
for every family $\{L_t\}_{t\in U}$ of linear subspaces of dimension $n-k$
with $L_0=L$
there is an open neighborhood $W$ of $0$
such that $\lct(I|_{L_{t}})\geq\lct(I|_{L_0})$ for every $t\in W$.
So $\lct^{(n-k)}(I|_L)$ is well-defined and characterized as maximal
possible one,
despite of the fact that the isomorphism classes of $I|_{L_t}$ may vary
along $t$.

When $L$ is the zero set of $h_1,\dots,h_{k}$, then
$\lct^{(n-k)}(I)$ is the log canonical threshold of the ideal generated by
the image of $I$ in $\mathcal O_n/\langle h_1,\dots,h_k\rangle$.
By Proposition 4.5 of \cite{Mustata2002}
(or Property 1.17 of \cite{MustataIMPANGA}), we have
\begin{equation}\label{chainlct}
\lct^{(1)}(I)\leq \lct^{(2)}(I) \leq \cdots \leq \lct^{(n)}(I)
\end{equation}
We know that $\lct^{(n)}(I)=\lct(I)$ and $\lct^{(1)}(I)=1/\ord(I)$ are
bi-Lipschitz invariant.
So it is natural to ask the following
\begin{ques}
Is $\lct^*(I)=(\lct^{(n)}(I),\,\lct^{(n-1)}(I),\,\dots,\,\lct^{(1)}(I))$
a bi-Lipschitz invariant?
\end{ques}
Theorem \ref{LCTLojExp} has the following analogy for $\lct^{(k)}(I)$.
\begin{thm}\label{lct^kLojExp}
Let $I$ be an ideal of $\O_n$ such that $V(I)\subseteq V(x_1\cdots x_n)$. Then
$$
1-\frac{k}n \leq {\lct^{(n-k)}(I)}\LL_{x_1\cdots x_n}^{(n-k)}(I)
$$
for all $k=0,1,\dots, n-1$.
\end{thm}

\begin{proof}
Let $L$ be a linear $(n-k)$-dimensional subspace of $\C^n$.
Assume that $I$ is generated by $f_1$, \dots, $f_m$ and
set $f=(f_1,\dots,f_m)$.
Let $H_i=\{h_i=0\}$ denote a generic hyperplane of $\C^n$ through $0$
so that $L=H_1\cap\cdots\cap H_k$.
Let $\omega$ denote an $(n-k)$-form with
$dx_1\wedge\cdots\wedge dx_n=dh_1\wedge\cdots\wedge dh_k\wedge\omega$.
Let $\pi:M\to\C^n$ denote the blow up at the origin and let
$h_i'$ denote the strict transform of $h_i$.
Set $x_1=u_1$ and $x_i=u_1u_i$ ($i=2,\dots,n$).
Since $h_i=u_1h_i'$, then
$$
dh_i=d(u_1h_i')=u_1dh_i'+h_i'du_1=u_1dh_i'
$$
on the set defined by $h_i'=0$.
Let $\omega'$ denote an $(n-k)$-form with
$du_1\wedge\cdots\wedge du_n=dh'\wedge\omega'$.
Since $L$ is generic,
the strict transform $L'$ of $L$ and the zeros of $u_i$ ($i=2,\dots,n$)
form a normal crossing variety.
Since
\begin{align*}
(u_1dh'_1)\wedge\cdots\wedge(u_1dh'_k)\wedge \omega
=&dh_1\wedge\cdots dh_k\wedge \omega\\
=&dx_1\wedge\dots\wedge dx_n\\
=&u_1^{n-1}du_1\wedge\cdots\wedge du_{n} \qquad
\textrm{ on $L'$,}
\end{align*}
we may assume that $\omega=u_1^{n-k-1}\omega'$ on $L'$.
If $|x_1\cdots x_n|^\theta\lesssim\|f\|$ on $L$,
we have
\begin{align*}
\int_{K\cap L}\|f\|^{-2s}\frac{\omega\wedge\bar{\omega}}{\sqrt{-1}^{n-k}}
\lesssim&\int_{K\cap L}|x_1\cdots x_n|^{-2\theta s}
\frac{\omega\wedge\bar{\omega}}
{\sqrt{-1}^{n-k}}\\
=&\int_{\pi^{-1}(K)\cap L'}
|u_1^{n}u_2\cdots u_n|^{-2\theta s}|u_1|^{2(n-k-1)}
\frac{\omega'\wedge\bar{\omega}'}{\sqrt{-1}^{n-k}}
\\
=&\int_{\pi^{-1}(K)\cap L'}
|u_1|^{-2(n\theta s-n+k+1)}|u_2\cdots u_n|^{-2\theta s}
\frac{\omega'\wedge\bar{\omega}'}{\sqrt{-1}^{n-k}}
\end{align*}
which is integrable whenever $n\theta s<n-k$. So we have that
$s<(1-\frac{k}n)/\mathcal L_{x_1\cdots x_n}^{(n-k)}(I)$ implies
$s<\lct^{(n-k)}(I)$,
and we have
$$
1-\frac{k}n\leq{\lct^{(n-k)}(I)}\mathcal L_{x_1\cdots x_n}^{(n-k)}(I).
$$
\end{proof}
We close the paper to show a closed formula for $\lct^{(k)}(I)$ when
$\overline I$ is monomial.

\begin{thm}\label{mixlct}
Let $I$ be an ideal of $\O_n$ such that $\overline{I}$ is a monomial
 ideal.
Then
\begin{align*}
\lct^{(k)}(I)
=&\min\Bigl\{\frac{\sum_ia_i-(n-k)a_{\min}}{\ell(a,I)}:a\in S^{(k)}\Bigr\}
\\
=&\inf\Bigl\{\frac{\sum_ia_i-(n-k)}{\ell(a, I)}
: a\in S^{(k)}\cap A\Bigr\}
\end{align*}
where $A=\{a=(a_1,\dots,a_n):\min\{a_1,\dots,a_n\}=1\}$, for all $k\in\{1,\dots, n\}$.
\end{thm}

\begin{proof}
We may assume that $I$ is a monomial ideal.
We consider a toric modification $\sigma:X\to\C^n$
which dominate the blowing up at the origin.
There is a coordinate system $(y_1,\dots,y_n)$ so that $\sigma$ is
 expressed by
$$
x_i=y_1^{a^1_i}\cdots y_n^{a^n_i} \qquad(a^j_i\in\Z, \ i=1,\dots,n).
$$
Then we have $h_i=y_1^{a^1_{\min}}\cdots y_n^{a^n_{\min}}\tilde{h}_i$
where $\tilde{h}_i$ denotes the strict transform of $h_i$ by $\sigma$.
So we have
$$
dh_i=y_1^{a^1_{\min}}\cdots y_n^{a^n_{\min}} d\tilde{h}_i
$$
on the set defined by $\tilde{h}_i=0$.
Since
\begin{align*}
\Bigl(
\wedge_{i=1}^k(y_1^{a^1_{\min}}\cdots
 y_n^{a^n_{\min}}d\tilde{h}_i)\Bigr)\wedge
\omega
=&dh_1\wedge\cdots\wedge dh_k\wedge\omega\\
=&dx_1\wedge\cdots\wedge dx_n\\
=&y_1^{\sum_ia^1_i-1}\cdots y_n^{\sum_ia^n_i-1}dy_1\wedge\cdots\wedge dy_n
\end{align*}
we obtain that
$$
\omega=
y_1^{\sum_ia^1_i-ka^1_{\min}-1}
\cdots y_n^{\sum_ia^n_i-ka^n_{\min}-1}\tilde{\omega}
$$
where $\tilde{\omega}$ is a holomorphic $(n-k)$-form
which does not vanish on the strict transform
$\tilde{L}$ of $L$ by $\sigma$ with
$$
dy_1\wedge\cdots\wedge dy_n=
d\tilde{h}_1\wedge\cdots\wedge d\tilde{h}_k\wedge\tilde{\omega}.
$$
Since $L$ is generic, $\tilde{L}$ and the zeros of $y_j$ form a
 normal crossing variety and we conclude that
$$
\lct^{(n-k)}(I)
=
\min\Bigl\{
\frac{\sum_i a_i-ka_{\min}}{\ell(a, I)}:a\in S^{(n-k)}
\Bigr\}.
$$
We complete the proof by replacing $k$ by $n-k$.
\end{proof}


\end{document}